\documentclass[10pt]{article}
\usepackage[utf8]{inputenc}
 
\usepackage[ruled,linesnumbered]{algorithm2e}

\usepackage{amssymb}
\usepackage{amsmath}
\usepackage{amsthm}
\usepackage{url}
\usepackage{mathrsfs}
\usepackage{pdfpages}
\usepackage{epigraph}
\usepackage{tikz-cd}

\newtheorem{lem}{Lemma}

\newtheorem{mytheor}{Theorem}

\newtheorem{theor}{Theorem}

\newtheorem{rem}{Remark}

\newcommand{\N}{\mathbb{N}}
 
\newcommand{\Z}{\mathbb{Z}}

\DeclareMathOperator{\pr}{pr}

\DeclareMathOperator{\rank}{rk}

\begin{document}

\author{Andrei Alpeev  \footnote{École Normale Supérieure, 45 Rue d'Ulm, Paris, alpeevandrey@gmail.com}}
\title{Secret sharing on the Poisson-Furstenberg boundary I}

\maketitle
\begin{abstract}
Recently, Vadim Kaimanovich presented a particular example of a measure on a product of two standard lamplighter groups such that the Poisson boundary of the induced random walk is non-trivial, but the boundary on the marginals is trivial. This was surprising since such behavior is not possible for measures of finite entropy. As we show in this paper, this secret-sharing phenomenon is possible precisely for pairs of amenable groups with non-trivial ICC-factors.
\end{abstract}

\section{Introduction}

Given a countable group $\Gamma$ and a measure $\nu$ on that group, we can construct a random walk on the group. First, consider the i.i.d. process $(X_n)_{n\in \N}$ where each $X_i$ has distribution $\nu$ and then define random variables $Z_n = X_1 \cdot \ldots \cdot X_n$. The process $(Y_n)_{n \in \N}$ is called the $\nu$-random walk, its space of trajectories $\Omega = {\Gamma}^{\N}$. The {\em Poisson (or Poisson-Furstenberg) boundary} ot this random walk is defined as the space $\partial(\Gamma, \nu)$ of ergodic components of $\Omega$ under the natural time-shift action. A measure $\nu$ is called {\em non-degenerate} if its support generates $G$ as a semigroup. Under this requirement, another object turns out to be isomorphic to the Poisson boundary: the {\em tail boundary}, that is the factor of $\Omega$ that corresponds to the tail subalgebra. 
We will casually conflate subalgebras and corresponding factor.
In what follows it is always required that measure $\nu$ is non-degenerate so we will use these objects interchangeably and sometimes simply call them {\em the boundary}. We will say that the boundary is trivial if the corresponding subalgebra is mod-0 trivial (equivalently, corresponding factor is essentially a one-point set), we will also call a measure $\nu$ on a group $G$ a {\em Liouville measure} exactly when the boundary $\partial(G, \nu)$ is trivial.

There is a natural action of $\Gamma$ on the space $\Omega$:
$$g \cdot (\omega_1, \omega_2, \ldots) = (g \omega_1, g \omega_2, \ldots ),$$
where $g \in \Gamma$ and $(\omega_1, \omega_2, \ldots) \in \Omega$.
It turns out that we can restrict this action to the boundary and get an action with a quasi-invariant measure.  A factorization $\varphi : \Gamma_1 \to \Gamma_2$ induces an equivariant map from $\partial(\Gamma_1, \nu)$ to $\partial(\Gamma, \varphi(\nu))$, and in particular, the boundary $\partial(\Gamma_2, \varphi(\nu))$ is a subalgebra of $\partial(\Gamma_1, \nu)$
Now, Assume that $\Gamma = G_1 \times G_2$ and $\nu$ is a non-degenerate measure on $G_1 \times G_2$

We have the following natural factorizations:

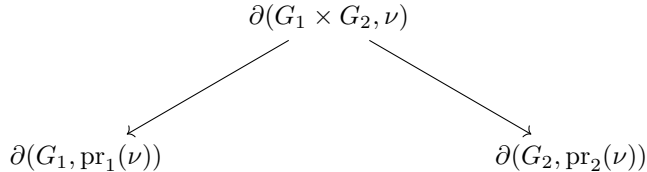
\[
\begin{tikzcd}
	&\partial(G_1 \times G_2, \nu)\ar[ddl]\ar[ddr]&\\\\
	\partial(G_1, \pr_1(\nu))&&\partial(G_2, \pr_2(\nu))
\end{tikzcd}
\]

A natural question is whether the boundary on $G_1 \times G_2$ is generated by the boundaries on $G_1$ and $G_2$. It turns out that if we require $\nu$ to have finite Shannon entropy, then the answer is affirmative, see \cite[Claim 4.1]{ErFr22} for the proof. This is done using Kaimanovich's conditional entropy criterion \cite[Section 4.6]{Ka00}. It might be that the finite entropy condition is a simple technicality and is not really required. After all, the statement is simply about equality between two $\sigma$-algebras. Surprisingly, this is not the case, recently Kaimanovich presented the following example:

\begin{theor}[\cite{Ka24}]
Let $G_1 = G_2 =\Z/2\Z \wr \Z$ be two copies of the standard lamplighter group.
There is a non-degenerate measure $\nu$ on the product $G_1 \times G_2$ such that the boundary $\partial(G_1 \times G_2, \nu)$ is non-trivial, but boundaries of projection $\partial(G_i, \pr_i(\nu))$, $i = 1,2$ are trivial.
\end{theor}

This presents a kind of a secret sharing phenomenon. We have two random processes. Each has a trivial tail and so in some sense does not retain asymptotically any information in its distant future. On the other hand, the distant future of the combined process retains quite a lot of information. 

A natural question arises: under which requirements do we have such an examples. The necessary condition is that both groups are not hyper-FC, see \cite{Ja}. 
A group is called {\em hyper-FC} if it has no ICC factors.
And a group has {\em ICC (infinite conjugacy class) property} if all its non-trivial elements have infinite conjugacy classes. 
See section \ref{sec: FC} for the discussion. It will suffice to say now that finitely-generated groups are hyper-FC exactly when they are virtually nilpotent. Also, it is necessary that both groups are amenable, since this is a requirement for a group to have a non-degenerate Liouville measure, by theorem of Furstenberg \cite[Section 9]{Fu74}. Are these requirements sufficient for the secret-sharing phenomenon? As we will show in this paper, they are. As the main result we establish the following:

\begin{mytheor}\label{thm: simple pair}
Let $G_1$ and $G_2$ be two countable amenable ICC groups. There is a measure of full support $\nu$ on $\Gamma = G_1 \times G_2$ such that the Poisson-Furstenberg boundary $\partial(\Gamma, \nu)$ is nontrivial, moreover $\Gamma$ acts freely on this boundary, and the boundaries of projections $\partial(G_1, \pr_1(\nu))$, $\partial(G_2, \pr_2(\nu))$ are trivial. Also, measure $\nu$ could be taken symmetric.
\end{mytheor}

Using standard techniques we deduce:

\begin{mytheor}\label{thm: main}
Let $G_1$ and $G_2$ be two countable groups. 
the following are equivalent
\begin{enumerate}
\item There is a non-degenerate measure $\nu$ on $\Gamma = G_1 \times G_2$ such that the Poisson-Furstenberg boundary $\partial(\Gamma, \nu)$ is nontrivial, and the boundaries of projections $\partial(G_1, \pr_1(\nu))$, $\partial(G_2, \pr_2(\nu))$ are trivial.
\item $G_1$ and $G_2$ are amenable groups with non-trivial ICC factor-groups.
\end{enumerate}
\end{mytheor}

\begin{proof}
That 2 implies 1 constitute the bulk of this paper. In Section \ref{sec: nonsymmmetric} we show the construction without the symmetricity condition and in Section \ref{sec: symmetric} we demonstrate a slight modification to get a symmetric measure.

We will prove that 1 implies 2. Assume that $G_1$ has no non-trivial ICC factors. This implies that $G_1$ is a hyper FC-central group an so, by Lemma \ref{lem: FC-isomorphism}, factorization $G_1 \times G_2 \to G_2$ induces an equivariant isomorphism of boundaries $\partial(G_1 \times G_2 , \nu)$  and $\partial(G_2, \pr_2(\nu))$. We get a contradiction.
\end{proof}

Interestingly, the construction of Kaimanovich gives an asymmetric measure without a clear way to avoid this. Our examples demonstrates that asymmetricity plays no role in this story.

The main construction relies on the ideas from my previous work \cite{A21} where I constructed examples of measures with non-trivial left and trivial right tail boundaries for all amenable groups with non-trivial ICC factors. In turn, that paper is based upon the breakthrough construction \cite{FHTF19} by  Frisch, Hartman, Tamuz and Ferdowsi of a measure with non-trivial boundary for any group with non-trivial ICC factor, but our presentation relies on some clarifications made in the paper \cite{ErKa19} by Erschler and Kaimanovich. We also inject the main idea of the Kaimanovich-Vershik \cite{KaVe83} and Rosenblatt \cite{Ro81} construction of a measure with trivial Poisson boundary on an arbitrary amenable group.

{\bf Acknowledgements.} I'm thankful to Vadim Kaimanovich and Anna Erschler for helpful discussions on the topic. I'm grateful to my late advisor Anatoly Moiseevich Vershik who introduces me to the topic of boundary theory of random walks.

\section{Criteria for triviality and non-triviality of the Poisson-Furstenberg boundary}

Let $G$ be group an let $\nu$ be a measure of full support on $G$.
We first introduce two devices that will help us establish non-triviality and triviality of the Poisson-Furstenberg boundary.

The following lemma is essentially an abstract form of the argument used in \cite{FHTF19} as presented in \cite{ErKa19}.
\begin{lem}\label{lem: main non-triviality}
Let $W_i, W'_i$, for $i \in \N$ be subsets of $G$ such that $W'_i \subset W_i$ and $W_i \cap W_j = \varnothing$ for $i \neq j$. 
For $g \in G$ denote $\rank(g)$ to be $i$ if $g \in W_i$ for some $i \in \N$ and $\rank(g) = 0$ otherwise.
Assume that there is a function $p: \bigcup_{i \in \N} W_i \to G$. We require that
\begin{enumerate}
\item if $p(g)$ is defined then $\rank(p(g)) < \rank(g)$;
\item for every $h \in G$ there is $N$ such that $h W'_i \subset W_i$,and $p(hw) = hp(w)$ for all $i > N$ and $w \in W_i$;
\item for almost every trajectory $\zeta = (z_1, z_2, \ldots)$ there is an $i_0$ such that for all  $i > i_0$ there is $j$ such that $z_i \in W'_j$;
\item for almost every trajectory $\zeta = (z_1, z_2, \ldots)$ there is an $i_0$ such that if $i > i_0$ then either $p(z_{i+1}) = p(z_i)$ or $p(z_{i+1}) = z_i$.
\item the sequence $\rank(z_i)$ is unbounded.
\end{enumerate} 
The above implies that the Poisson-Furstenberg boundary $\partial(G, \nu)$ is anon-trivial and moreover, the action of $G$ on said boundary is essentially-free.
\end{lem}

\begin{rem}
We may only require that $\nu$ is non-degenerate. The requirement that $\nu$ has full support allows us to avoid a minor technical difficulty that the distribution on the space of trajectories is not quasi-invariant with respect to the $G$-action. 
\end{rem}

\begin{proof}
For two finite or infinite sequences $\alpha$ and $\beta$ we write $\alpha \subseteq \beta$ if $\alpha$ is an initial segment of $\beta$. If $\alpha_1 \subseteq \alpha_2 \subseteq \alpha_3 \subseteq \ldots$ is a finite or infinite monotone collection of sequences, we denote $\bigcup_i \alpha_i$ the minimal sequence for which all $\alpha_i$ are initial segments.
	
First, we define a function $t: G \to G^{<\N}$ from the group to the space of finite sequences of group elements as a sequence of all values of iterated application of $p$ in reverse order: $(p^n(g), p^{n-1}(g),\ldots, p^2(g), p(g))$, where $p^n(g) \notin\bigcup_{i \in \N} W_i$. By requirement 1 this is indeed a finite sequence. 
	
Let $\Omega = G^{\N}$ be the space of all the trajectories of the $\nu$-random walk.
We start by defining a function $\tau : \Omega \to G^{\N}$ in the following way. For almost every $\zeta = (z_1, z_2, \ldots) \in \Omega$ there is $i_0$ such that $t(z_i)$ for $i > i_0$ form a nested collection of sequences. We set $\tau(\zeta) = \bigcup_{i>i_0} t(z_i)$. Now it is easy to check, using requirement 4 and 5, that $\tau(\zeta)$ is measurable with respect to the tail subalgebra and $\tau(\zeta)$ is an infinite subsequence of $\zeta$. This actually implies that $\tau$ is an isomorphism since for any tail-measurable bounded function $f$, $f(\zeta) = \lim_{i \in \N} \bar{f}(z_i) = \lim_{i \in \N} \bar{f}((\tau(\zeta))_i)$, where $\bar{f}$ is the harmonic function corresponding to $f$. Now, let us how that the action of $G$ on the boundary is essentially free. Let us prove that for every $h \in G$ we have that $\tau(\zeta) \neq \tau(h \zeta)$ for almost every $\zeta$. We notice that, since the action of $G$ on $\Omega$ is quasi-invariant, all ``almost all'' statements hold for almost all $\zeta$ and $h \cdot \zeta$. Now using requirements 2 and 3 for $\zeta$ we conclude, that if $\tau(\zeta) = (\gamma_1, \gamma_2, \ldots)$ and then there is some $m$ such that $\tau(h \cdot \zeta) = (\lambda_1, \ldots, \lambda_n, h\gamma_m, h \gamma_{m+1}, \ldots)$. It is now easy to check that $\tau(\zeta)$ and $\tau(h \cdot \zeta)$ are distinct, since, for example there is $W_i$ such that the unique element of the sequence $\tau(\zeta)$ in $W_i$ is some $\gamma_j$ with $j > m$ and the unique element of the sequence $\tau(h \cdot \zeta)$ in $W_i$ is $h\gamma_j$.
\end{proof}

\begin{lem}
If for every $h \in G$ we have 
$$\lim_{n\to\infty}\lVert h * \nu^{*n} - \nu^{*n}\rVert = 0,$$
then the Poisson-Furstenberg boundary is trivial.
\end{lem}

We will need the following obvious variation:
\begin{lem}\label{lem: kv triviality}
Let $S$ any generating set for $G$. If for every $h \in S$ we have 
$$\lim_{n\to\infty}\lVert h * \nu^{*n} - \nu^{*n}\rVert = 0,$$
then the Poisson-Furstenberg boundary is trivial.
\end{lem}

\section{Hyper FC-center of the direct sum}\label{sec: FC}
Let $G$ be a group. It has a maximal ICC quotient, that is an ICC group $\Gamma$ together with an epimorphism $\phi : G \to \Gamma$ such that for if $\Gamma'$ is an ICC group and $\phi': G \to \Gamma'$ is an epimorhism, then there is a homomorphism $\phi'' : \Gamma \to \Gamma'$ such that $\phi' = \phi'' \circ \phi$.
In order to see this we first define the FC-center of a group as the subgroup of all elements that have finite conjugacy classes. trivially, these elements should be in the kernel of any epimorphism to an ICC group. We define the following chain of groups: 
$$ G^0 \to G^1 \to \ldots,$$
indexed by ordinal numbers in such a way that $G^0 = G$, $G^{\alpha+1}$ is the factor of $G^\alpha$ by the FC-center of $G^{\alpha}$, and $G^{\beta}$ for a limit ordinal $\beta$ is the limit of the diagram formed by all $G^{\alpha}$ with $\alpha < \beta$. This sequence trivially stabilizes, at a (possibly trivial group) $G^{\alpha_0}$. It is easy to see that $G^{\alpha_0}$ is an ICC group.
It is also easy to show inductively that any epimorphism from $G$ to an ICC group factors through every $G^{\alpha}$, which implies that $G^{\alpha_0}$ is indeed the maximal ICC factor of $G$. The hyper FC-center of a group is defined as the kernel of the canonical epimorphism to its maximal ICC factor. 
Let $G_i$, $i \in I$ be a collection of groups. 
\begin{lem}\label{lem: icc of product}
The maximal ICC factor of a direct sum of group $\bigoplus_{i \in I} G_i$ is the direct sum of their maximal ICC factors $\bigoplus_{i \in I} \Gamma_i$. 
\end{lem}
\begin{proof}
We first observe that the direct sum of ICC factors is an ICC factor of $\bigoplus_{i \in I} G_i$. To see that this is the maximal ICC factor we note that the FC-center of a direct sum of groups is the direct sum of their FC-centers. inductively, we observe that the hyper FC-center is the direct sum of hyper FC-centers.
\end{proof}
For the following see Jaworski \cite[Lemma 4.7]{Ja}:

\begin{lem}\label{lem: FC-isomorphism}
Let $G$ be a group and $\Gamma$ be its factor with the canonical epimorphism $\varphi: G \to \Gamma$ such that $\ker \varphi$ is a subset of the hyper-FC center of $G$. Then for every non-degenerate measure $\nu$ on $G$, $\varphi$ induces the natural isomorphism between Poisson-Furstenberg boundaries $\partial(G ,\nu)$ and $\partial(\Gamma, \varphi(\nu))$.
\end{lem}

\section{Example of asymmetrical measure on the product of two groups}\label{sec: nonsymmmetric}




Let $A$ be a subset of a group $G$ and $\delta > 0$ be a number. We say that a subset $F$ of $G$ is $(A,\delta)$-invariant if it is finite, nonempty and $\lvert AF \setminus F\rvert < \delta \lvert F\rvert$.
Let $G$ be a group and $A$ be finite subset. We say that $b \in G$ is an $A$-switcher if $A \cap AbA = \varnothing$ and for every $a'_1, a'_2, a''_1, a''_2 \in A$ the equality $a'_1 b a'_2 = a''_1 b a''_2$ implies that $a'_1 = a''_1$ and $a'_2 = a''_2$.

The following lemma is a simplified version of \cite[Proposition 2.5]{FHTF19} and \cite[Poposition 4.25]{ErKa19}, 
\begin{lem}
Let $G$ be a non-trivial ICC group. For every finite subset $A$ of $G$ there is an $A$-switcher. 
\end{lem}

First let us fix a sequence of pairs $(c_{1,i},c_{2,i})_{i \in \N}$, $c_{j,i} \in G_j$ for $i \in \N$ and $j = 1,2$, that enumerates all elements of $G_1 \times G_2$. 

We set $A_{1,1} = \lbrace 1_{G_1}\rbrace$ and$A_{2,1} = \lbrace 1_{G_3}\rbrace$. For each $i \geq 1$ and $j \in \lbrace 1, 2\rbrace$:
\begin{enumerate}
\item let $F_{j,i}$ be an $(A_{j,i}^{i+1}, 1/i)$-invariant subset of $G_j$;
\item let $S_{j,i}$ be any subset of $G_j$ such that $\lvert S_{j,i}\rvert = \lvert F_{2-j,i} \rvert$;
\item let $b'_{j,i} \in G_j$ be an $(A_{j,i} \cup S_{j,i} \cup F_{j,i})^{i+2}$ - switcher;
\item let $b''_{j,i} \in G_j$ be an $(A_{j,i} \cup S_{j,i} \cup F_{j,i} \cup \lbrace b'_{j,i}\rbrace)^{2i+8}$-switcher;
\item let $A_{j,i+1} = A_{j,i} \cup F_{j,i}b'_{j,i} S_{j,i} b''_{j,i} \cup \lbrace c_{j,i}\rbrace$;
\end{enumerate} 

We also fix arbitrary bijections $\psi_{}: F_{2-j,i} \to S_{2-j,i}$, for $i \in \N$ and $j = 1,2$.

Let $(k_i)_{i \in \N}$, $k_i \in \N$ be a sequence. We say that $i$ is a record-time for the sequence if $k_i > k_j$ for $j < i$. We say that $i$ is a non-strict record-time if $k_i \geq k_j$ for all $j < i$.
We also will say that the pair $(i, k_i)$ is a a record and sometime that simply $k_i$ is a record. We say that the record $k_i$ (and the corresponding record-time $i$) is simple if $k_i < k_j$ for all non-strict record-times $j > i$.

Now let $K$ be an $\N$-valued random variable such that $\Pr(K=k) = k^{5/4}/c$ ($c$ is a normalization constant).
Let $(K_i)_{i \in \N}$ be the process of i.i.d. copies of $K$.

\begin{lem}
For almost every trajectory of the process $k_i$ there is a number $i_0$ such that 
\begin{enumerate}
\item all record-times $i \geq i_0$ are simple;
\item $\max\lbrace k_1, \ldots, k_i\rbrace > i$ for all $i \geq i_0$.
\end{enumerate}
\end{lem}

We construct a coupled with $K$ random variable $Y$ that takes two values: ``blue'' with probability $1-2^{-K}$ and ``red'' with probability $2^{-K}$.

Consider the process $(K_i, Y_i)_{i \in \N}$ of i.i.d. copies of the pair $(K,Y)$.
We say that a trajectory $(k_i,y_i)$ of this process stabilizes if there is $i_0$(we will call it a stabilization time) such that for all $i > i_0$:
\begin{enumerate}
\item a maximal value $k_m$ among $k_1, \ldots, k_i$ is a simple record;
\item $\max\lbrace k_1, \ldots, k_i\rbrace > i$ for all $i \geq i_0$;
\item $y_m = \text{``blue''}$ for the maximal record-time $m$ on the segment $1, \ldots, i$.
\end{enumerate}

It follows from the previous lemma and the Borel-Cantelli lemma that

\begin{lem}\label{lem: stabilization}
Almost every trajectory of the process $(K_i, Y_i)_{i \in \N}$ stabilizes.
\end{lem}

Let $X$ be $G_1 \times G_2$-valued random variable coupled to the pair $(K,Y)$ in the following way. Assume $K=k$. If $Y = \text{``red''}$, then $X = (c_{1,k}, c_{c_k})$, else if $Y = \text{``blue''}$, then we do the following.
Let $f_{1,k}$ and $f_{2,k}$ be independently and uniformly randomly distributed in $F_{1,k}$ and $F_{2,k}$ respectively. We put $X = (f_{1,k}b'_{1,k}\psi_{1,k}(f_{2,k})b''_{1,k},f_{2,k}b'_{2,k}\psi_{2,k}(f_{1,k})b''_{2,k})$. It is worth observing, that for each $j=1,2$, conditionally to $K = k$ and $Y = \text{``blue''}$, the distribution $\pr_j(X)$ is the same as of $f_{j,k}b'_{j,k}s_{j,k}b''_{j,k}$, where $f_{j,k}$ is distributed uniformly in $F_{j,k}$ and $s_{j_k}$ is independently uniformly distributed in $S_{j,k}$.
Now, the distribution of $X$ is the measure $\nu$ on $G_1 \times G_2$ we wanted to construct, but in the proofs it would be convenient for us to consider the i.i.d. process $(X_i)$ as a marginal of the i.i.d. triple-process $(K_i, Y_i, X_i)$. 

\begin{lem}\label{lem: simple pair trivial}
Measure $\pr_j(\nu)$ on $G_j$ is Liouville.
\end{lem}
\begin{proof}
Without loss of generality we will prove the lemma for $\pr_1$.
Fix an element $g \in G_1$ and $\varepsilon > 0$ Let $n_0$ be such that $g \in A_{1,n_0}$ and $n_0 > 1/\varepsilon$.
Consider the i.i.d. process of triples $(K_i, Y_i, X_i)$.
By Lemma \ref{lem: stabilization}, with probability bigger than $1-\varepsilon$ we have that for big enough $n > n_0$, the maximal value of $k_m$ for $m = 1, \ldots, n$ is unique on that segment, $k_m > n$ and the corresponding $y_m$ is ``blue''. 
Now condition our process to the latter requirements, and also fixing the values of $X_i$ for $i \neq m$, we get that the conditional distribution of
$$\pr_1(X_1) \pr_1(X_2) \ldots  \pr(X_{m-1})\pr(X_m)\pr(X_{m+1}) \ldots \pr_1(X_n)$$

is $\kappa = \pr_1(x_1) \pr_2(x_2) \ldots \pr_1(x_{m-1})\lambda_{F_{1,m}} b'_{1,m} \lambda_{S_{1,m}} b''_{1,m} \pr_1(x_{m+1} \ldots \pr_1{x_n})$. 
Now, by assumption, we get that $\pr_1(x_i) \in A_{m}$ for $i = 1, \ldots m-1, m+1 \ldots, n$. So $\kappa = q' \lambda_{F_{1,m}} b'_{1,m} \lambda_{S_{1,m}} b''_{1,m}q''$, where $q' \in A_{1,m}^{n-1}$. This implies that $\lVert g * \kappa  - \kappa \lVert < 2/m < 2\varepsilon$ (since $F_{1,m}$ is by construction $(A^m_{2,m},1/m)$-invariant). The latter implies  that $\lVert g * \pr_1(\nu^{*n})-\pr_1(\nu^{*n})\rVert < 4 \varepsilon$. 
\end{proof}

Now we will prove that $G_1 \times G_2$ acts freely on the Poisson-Furstenberg boundary $\partial(G_1 \times G_2, \nu)$.



Let us define subsets of $G_1 \times G_2$:
$$W_i = \lbrace (q'_1 f_1 b'_{1,i} \psi_{1,i}(f_2)b''_{1,i} q''_1,\quad q'_2 f_2 b'_{2,i} \psi_{2,i}(f_1)b''_{2,i} q''_2)\rbrace,$$
where $q'_j \in A^{i+1}_{j,i}$, $f_j \in F_{j,i}$, $q''_j \in A^i_{j,i}$,
and 
$$W'_i = \lbrace (q'_1 f_1 b'_{1,i} \psi_{1,i}(f_2)b''_{1,i} q''_1,\quad q'_2 f_2 b'_{2,i} \psi_{2,i}(f_1)b''_{2,i} q''_2)\rbrace,$$
where $q'_j \in A^{i}_{j,i}$, $f_j \in F_{j,i}$, $q''_j \in A^i_{j,i}$.

\begin{lem}\label{lem: simple pair sets}
Sets $W_i$ are pairwise disjoint, and $W'_i \subset W_i$ for all $i \in \N$.
\end{lem}
\begin{proof}
The second claim is obvious from the fact that $A_{j,i}$ contains the group identity.

Consider $l < r$. We note that $\pr_1(W_l)$ is a subset of $A_{1,l+1}^{2l+2} \subset A_{1,r}^{2r+2}$. Also, $\pr_1(W_r)$ is a subset of 
\begin{multline*}
A^{r+1}_{1,r}F_{j,r}b'_{1,r}S_{1,r}b''_{1,r}A^r_r \\\subset (A_{1,r} \cup S_{1,r} \cup F_{1,r} \cup \lbrace b'_{1,r}\rbrace)^{r+4}b''_{1,r}(A_{1,r} \cup S_{1,r} \cup F_{1,r} \cup \lbrace b'_{1,r}\rbrace)^r.
\end{multline*}
Now, by construction, $b''_{1,r}$ is an
$(A_{1,r} \cup S_{1,r} \cup F_{1,r} \cup \lbrace b'_{1,r}\rbrace)^{2r+8}$-switcher, so sets 
$$(A_{1,r} \cup S_{1,r} \cup F_{1,r} \cup \lbrace b'_{1,r}\rbrace)^{2r+8} b''_{1,r} (A_{1,r} \cup S_{1,r} \cup F_{1,r} \cup \lbrace b'_{1,r}\rbrace)^{2r+8}$$
and
$$(A_{1,r} \cup S_{1,r} \cup F_{1,r} \cup \lbrace b'_{1,r}\rbrace)^{2r+8}$$ 
are disjoint. Hence sets $\pr_1(W_l)$ and $\pr_1(W_r)$ are disjoint (reminder, $1_{G_1} \in A_{1,r}$), and so $W_l$ and $W_r$ are disjoint.
\end{proof}

\begin{lem}\label{lem: simple pair decomposition}
If some $g \in G_1 \times G_2$ could be presented in a form
$$g = (q'_1 f_1 b'_{1,i} \psi_{1,i}(f_2)b''_{1,i} q''_1,\quad q'_2 f_2 b'_{2,i} \psi_{2,i}(f_1)b''_{2,i} q''_2),$$
where $q'_j \in A^{i+1}_{j,i}$, $f_j \in F_{j,i}$, $q''_j \in A^i_{j,i}$, for $j = 1,2,$
then this decomposition is unique.
\end{lem}
\begin{proof}
Assume that there is alternative decomposition:
$$g = (\bar{q}'_1 \bar{f}_1 \bar{b}'_{1,i} \psi_{1,i}(\bar{f}_2)\bar{b}''_{1,i} \bar{q}''_1,\quad \bar{q}'_2 \bar{f}_2 \bar{b}'_{2,i} \psi_{2,i}(\bar{f}_1)\bar{b}''_{2,i} \bar{q}''_2),$$
where $\bar{q}'_j \in A^{i+1}_{j,i}$, $\bar{f}_j \in F_{j,i}$, $\bar{q}''_j \in A^i_{j,i}$,
Since $b''_{j,i}$ is a switcher by construction, we get that $q''_j = \bar{q}''_j$ and $q'_j f_j b'_j \psi_{j,i}(f_{2-j}) = \bar{q}'_j \bar{f}_j \bar{b}'_j \psi_{j,i}(\bar{f}_{2-j})$, for $j = 1,2$. Now, using the fact that $b'_{j,i}$ is a switcher, we get that $\psi_{j,i}(f_{2-j}) = \psi_{j,i}(\bar{f}_{2-j})$ and $q'_j f_j  = \bar{q}'_j \bar{f}_j$, for $j=1,2$. 
We derive $f_{2-j} = \bar{f}_{2-j}$, i.e. $f_{j} = \bar{f}_{j}$, since $\psi_{j,i}$ are bijections from $F_{2-j_i}$ to $S_{j,i}$, for $j = 1,2$.
We conclude that $q'_j = \bar{q}'_j$, for $j =1,2$;
\end{proof}

Let us construct the map $p$ that will satisfy the requirements of Lemma \ref{lem: main non-triviality}. For $g \in G_1 \times G_2$ that could be presented as 
$$g = (q'_1 f_1 b'_{1,i} \psi_{1,i}(f_2)b''_{1,i} q''_1,\quad q'_2 f_2 b'_{2,i} \psi_{2,i}(f_1)b''_{2,i} q''_2),$$
where $q'_j \in A^{i+1}_{j,i}$, $f_j \in F_{j,i}$, $q''_j \in A^i_{j,i}$, for $j = 1,2$, we set 
$$p(g) = (q'_1 ,\quad q'_2),$$
It follows from Lemmata \ref{lem: simple pair sets} and \ref{lem: simple pair decomposition} that $p$ is a well-defined map on $\bigcup_i W_i$.

Remind that for $g \in G$ we define $\rank(g)$ to be the unique $i \in \N$ such that $ i \in W_i$ or $0$ if $g \notin \bigcup_i W_i$.

\begin{lem}\label{lem: simple pair p-comb}
\begin{enumerate}
\item $\rank(p(g)) < \rank(g)$ if $p(g)$ is defined;
\item $hw \in W_i$ for all $h \in A_{1,i} \times A_{2,i}$ and $w \in W'_i$;
\item $p(hw) = hp(w)$ for all $h \in A_{1,i} \times A_{2,i}$ and $w \in W'_i$.
\end{enumerate} 
\end{lem} 
\begin{proof}
For (1) assume $g \in W_i$, take $r \geq i$. We have $\pr_1(p(g)) \in A_{1,i}^{i+1} \subset A^{r+1}_{1,r}$ and $\pr(W_r) = A^{r+1}_{1,r} F_{1,r} b'_{1,r} S_{1,r} b''_{1,r}$. We conclude that $\pr_1(p(g)) \notin \pr(W_r)$ since by construction $b'_{1,r}$ is a switcher. Hence $p(g) \notin W_r$ and so $\rank(p(g)) \neq r$.

For (2) we simply note that $W_i = (A_{1,i} \times A_{2,i}) W'_i$.

For (3) we observe that 
$$w = (q'_1 f_1 b'_{1,i} \psi_{1,i}(f_2)b''_{1,i} q''_1,\quad q'_2 f_2 b'_{2,i} \psi_{2,i}(f_1)b''_{2,i} q''_2),$$
where $q'_j \in A^{i}_{j,i}$, $f_j \in F_{j,i}$, $q''_j \in A^i_{j,i}$,
so $p(w) = (q'_1, q'_2)$. 
Let $h = (h_1, h_2) \in A_{1,i} \times A_{2,i}$.
Now 
$$hw = (\bar{q}'_1 f_1 b'_{1,i} \psi_{1,i}(f_2)b''_{1,i} q''_1,\quad \bar{q}'_2 f_2 b'_{2,i} \psi_{2,i}(f_1)b''_{2,i} q''_2),$$
where $\bar{q}'_{j,i} = h_j q'_{j,i} \in A^{i+1}_{j,i}$. By Lemma \ref{lem: simple pair decomposition} we conclude that 
$$p(hw) = (h_1 q'_{1,i}, h_2 q'_{2,i}) = h\, p(w).$$ 

Remind that to the i.i.d. process $(X_i)$ we associate the random-walk process $Z_i = X_1 \cdot \ldots \cdot X_i$.
\end{proof}

\begin{lem}\label{lem: simple pair p-prob}
For almost every realization $(k_n, y_n, x_n)$ of the triple process $(K_n, Y_n, X_n)$ (hence, for almost every realization of the process $(X_n)$) there is $N$ such that for all $n > N$
\begin{enumerate}
\item $z_n \in \bigcup_{r}W'_r$;
\item either $p(z_{n+1}) = p(z_{n})$ or $p(z_{n+1}) = z_n$,
\end{enumerate}
where $z_n = z_1 \cdot \ldots \cdot z_n$, for $j \in \N$.
\end{lem}
\begin{proof}
We take $N$ to be the stabilization time for the sequence $(k_n, y_n)$ (see Lemma \ref{lem: stabilization}. By definition, for every $n > N$, the maximal value $i = k_m$ among $k_1, \ldots , k_n$ is unique and $y_m = \text{``blue''}$, and $k_m > n$. It is easy to see that by construction we get that $z_n = x_1 \cdot \ldots \cdot x_m \cdot \ldots \cdot x_n$, so it could ve represented as:
$$z_n = (q'_1 f_1 b'_{1,i} \psi_{1,i}(f_2)b''_{1,i} q''_1,\quad q'_2 f_2 b'_{2,i} \psi_{2,i}(f_1)b''_{2,i} q''_2),$$ 
where $(q'_1, q'_2) = x_1 \cdot \ldots \cdot x_{m-1} \in (A_{1,i} \times A_{2,i})^{m-1}$, $f_j \in F_{j,i}$ for $j = 1,2$, and $(q''_1, q'_2) = x_{m+1} \cdot \ldots \cdot x_{n} \in (A_{1,i} \times A_{2,i})^{n-m}$. It is easy to see that $z_n$ belongs to  $W'_i$, hence (1) is proved.

Carry on with the same notation, consider $z_{n+1}$. We either have that $k_{n+1} < k_m = i$ (the record is no beaten) and we, can see that $z_{n+1}$ decomposes as 
$$z_n = (q'_1 f_1 b'_{1,i} \psi_{1,i}(f_2)b''_{1,i} \bar{q}''_1,\quad q'_2 f_2 b'_{2,i} \psi_{2,i}(f_1)b''_{2,i} \bar{q}''_2),$$
where $(\bar{q}''_1, \bar{q}''_2) = (q''_1, q''_2) \cdot (x_{n+1}, x_{n+1})$. In this case we (a posteriori) get that $n+1 < i$ (using the fact that the trajectory stabilized), and it is easy to check, applying Lemma \ref{lem: simple pair decomposition}, that $p(z_{n+1}) = (q'_1, q'_2) = p(z_n)$.

Possibility $k_{n+1} = k_m$ is excluded since $n$ is bigger than the stabilization time.

Now assume that $r = k_{n+1} > k_m = i$. Then $y_{n+1} = \text{``blue''}$ and $k_{n+1} > n+1$. Also note that $z_1, \ldots , z_n \in (A_{1,r} \times A_{2,r})$.
Now it is easy to check that $p(z_{n+1}) = z_n$.
\end{proof}

\begin{proof}[Proof of Theorem \ref{thm: simple pair} without symmetric requirement.]
We already proved that measure boundaries $\partial(G_j, \pr_j(\nu))$, $j = 1,2$ are trivial in Lemma \ref{lem: simple pair trivial}.
Lemmata \ref{lem: simple pair p-comb} and \ref{lem: simple pair p-prob} verify requirements of Lemma \ref{lem: main non-triviality}, so the Poisson-Furstenberg boundary $\partial(G_1 \times G_2, \nu)$ is non-trivial, and moreover $G_1 \times G_2$ acts essentially freely on it.
\end{proof}

\section{Example of symmetrical measure on the product of two groups}\label{sec: symmetric}
In this section we finish the proof of Theorem \ref{thm: simple pair} by providing an example of a symmetric measure of full support on the product $\Gamma = G_1 \times G_2$ of two amenable ICC groups such that the action of $\Gamma$ on the Poisson-Furstenberg boundary is non-trivial and the $G_j$-marginals have trivial boundaries. We will need a more versatile notion of a super-switching element to this end which was introduced initially in \cite{FHTF19}. 

We introduce a notation $A^{\pm} = A \cup A^{-1}$ for a subset $A$ of a group.

A super-switcher for a subset $A$ of a group is an element $b$ such that $A \cap AbA = \varnothing$, $A \cap Ab^{-1}A = \varnothing$ and
$a'b^{\sigma}a'' = \bar{a}'b^{\bar{\sigma}}\bar{a}''$ with $a', a'', \bar{a}', \bar{a}'' \in A$ and $\sigma, \bar{\sigma} = \pm1$ implies that $a' = \bar{a}'$,  $a'' = \bar{a}''$ and $b^{\sigma} = b^{\bar{\sigma}}$(naturally, this does not rule out the possibility $b = b^{-1}$).

The following lemma is proved in \cite[Proposition 2.5]{FHTF19} for amenable groups and in \cite[Proposition 4.19]{ErKa19} for arbitrary group:
\begin{lem}
There is a superswitcher for every finite subset $A$ of an ICC group $G$.
\end{lem}

The essence of the construction and the proof remains the same as in the previous one but we need to tweak it in order to make the measure symmetric. We will highlight the differences.

Let  $(c_1,n, c_2,i)_{n\ in \N}$ be any sequence of pairs that enumerate all elements of $G_1 \times G_2$. We set $A_{1,1} = \lbrace 1_{G_1}\rbrace$ and $A_{2,1} = \lbrace 1_{G_2}\rbrace$.

For each $i \geq 1$ and $j \in \lbrace 1, 2\rbrace$:
\begin{enumerate}
\item let $F_{j,i}$ be a $(A_{j,i}^{i+1}, 1/i)$-invariant subset of $G_j$;
\item let $S_{j,i}$ be any subset of $G_j$ such that $\lvert S_{j,i}\rvert = \lvert F_{2-j,i} \rvert$;
\item let $b'_{j,i} \in G_j$ be an $(A_{j,i} \cup S^{\pm}_{j,i} \cup F^{\pm}_{j,i})^{i+2}$ - superswitcher;
\item let $b''_{j,i} \in G_j$ be an $(A_{j,i} \cup S^{\pm}_{j,i} \cup F^{\pm}_{j,i} \cup \lbrace b'_{j,i}\rbrace^{\pm})^{2i+8}$-superswitcher;
\item let $A_{j,i+1} = A_{j,i} \cup (F_{j,i}b'_{j,i} S_{j,i} b''_{j,i} F_{j,i})^{\pm} \cup \lbrace c_{j,i}\rbrace^{\pm}$.
\end{enumerate} 

Note that sets $A_{1,i} \times A_{2,i}$ form a growing sequence of symmetric sets whose union is $\Gamma = G_1 \times G_2$. 

We fix arbitrary bijections $\psi_{j,i} : F_{2-j,i} \to S_{j,i}$, for $j = 1,2$ and $i \in \N$.


We reuse the pair of variables $K, Y$ from the previous construction. Now we construct a $\Gamma = G_1 \times G_2$-valued random variable $X$ coupled with $K,Y$. We first flip a coin to get a $\sigma = \pm1$. Fix $K= k$.
If $Y = \text{``red''}$, then we set $X = (c^{\sigma}_{1,K}, c^{\sigma}_{2,k})$. Else, we first pick uniformly and independently $f_j \in F_{j,k}$ (so, two independent random choices), and then we set 
$$X = (f_1 b'_1 \psi_{1,k}(f_2)b''_1, \quad f_2 b_2 \psi_{2,k}(f_1) b''_2)^{\sigma}.$$
Let $\nu$ be the support of random variable $X$. 

It is easy to see that $\nu$ is symmetric. The following lemma has an almost identical proof to that of Lemma \ref{lem: simple pair trivial}.

\begin{lem}
The boundary $\partial(G_j, \pr_j(\nu))$ is trivial for $j = 1,2$.
\end{lem}
\begin{proof}
Let $\mu = \pr_1{\nu}$. We will prove that for every $g \in G_1$
$$\lim_{n \to \infty} \lVert g * \mu^{*n} - \mu\rVert = 0.$$
Consider the partial trajectory of the process i.i.d. $(k_n, y_n, \sigma_n)$ coupled with the i.i.d. process $(X_n)_{n \in \N}$. 
Let $M$ be such that $g \in A_M$ and $1/M < \varepsilon$.
Let us take $m$ to be the first record-time that is bigger than the stabilization time and $M$ and such that $\sigma_m = +1$. Note that such $m$ exists with probability $1$. This way we defined a random value $T = m$ coupled with our process.  Let $N$ be such that $T < N$ with probability $1 - \varepsilon$.
For every $n > N$ we have a decomposition:
$$\mu^{*n} =  \eta + \sum_{q,m,q'} \alpha_{q,m,q'} \cdot q * \lambda_{F_m} * q',$$
where $\lambda_{F_m}$ denotes the uniform measure on $F_m$, the weight of $\eta$ is smaller than $\varepsilon$, $m > n$ $q \in A^{m}_{1,m}$, $q' \in G_1$ and $\alpha_{q,m,q'} \geq 0$. It is easily follows from the approximate invariance of $\lambda$ that 
$$\lVert g * \mu^{*n} - \mu^{*n} \rVert < 8 \varepsilon$$.
\end{proof}

Let us define for $i \in \N$ sets
$$W_i  = \lbrace(\,q'_1 (f_1 b'_1 \psi_{1,i}(f_2) b''_1)^{\sigma} q''_1, \quad q'_2 (f_2 b'_2 \psi_{2,i}(f_1) b''_2)^{\sigma} q''_2\,)\rbrace,$$
where $\sigma = \pm1$, $q'_j \in A_{j,i}^{i+1}$, $f_j \in F_{j,i}$, $q'' \in A_{j,i}^i$  for $j = 1,2$, and 
$$W'_i  = \lbrace(\, q'_1 f_1 (b'_1 \psi_{1,i}(f_2) b''_1)^{\sigma} q''_1, \quad q'_2 (f_2 b'_2 \psi_{2,i}(f_1) b''_2)^{\sigma} q''_2 \,\rbrace),$$
where $\sigma = \pm1$, $q'_j \in A_{j,i}^{i}$, $f_j \in F_{j,i}$, $q'' \in A_{j,i}^i$, for $j = 1,2$.
We trivially have $W_i = (A_{1,i} \times A_{2,i}) W'_i$ and $W'_i \subset W_i$. Also, using the definition of superswitcher $b''_i$, it is no hard to see that sets $W_i$ are pairwise disjoint.

Now we will prove the analog of the unique decomposition Lemma \ref{lem: simple pair decomposition}:

\begin{lem}
If for an element $g \in \Gamma = G_1 \times G_2$ we have
$$g = (\,q'_1 (f_1 b'_1 \psi_{1,i}(f_2) b''_1)^{\sigma} q''_1, \quad q'_2 (f_2 b'_2 \psi_{2,i}(f_1) b''_2)^{\sigma} q''_2\,),$$
where $q'_j \in A_{j,i}^{i+1}$, $f'_j, f''_j \in F_{j,i}$, $q'' \in A_{j,i}^i$, for $j = 1,2$, then this decomposition is unique.
\end{lem} 
\begin{proof}
Suppose we have an alternative decomposition:
$$g = (\bar{q}'_1 (\bar{f}_1 b'_1 \psi_{1,i}(\bar{f}_2) b''_1)^{\bar{\sigma}} \bar{q}''_1, \quad \bar{q}'_2 (\bar{f}_2 b'_2 \psi_{2,i}(\bar{f}_1) b''_2)^{\bar{\sigma}} \bar{q}''_2),$$
where $\bar{q}'_j \in A_{j,i}^{i+1}$, $\bar{f}_j \in F_{j,i}$, $\bar{q}'' \in A_{j,i}^i$, for $j = 1,2$.
First prove that $\sigma = \bar{\sigma}$. Assume the contrary, without loss of generality that $\sigma = -1$ and $\bar{\sigma} = +1$. We get that
$$q'_1 (f_1 b'_1 \psi_{1,i}(f_2) b''_1)^{\sigma} q''_1 = \bar{q}'_1 (\bar{f}_1 b'_1 \psi_{1,i}(\bar{f}_2) b''_1)^{\bar{\sigma}} \bar{q}''_1,$$ 
so
$$q'_1 (f_1 b'_1 \psi_{1,i}(f_2) b''_1)^{-1} q''_1 = \bar{q}'_1 (\bar{f}_1 b'_1 \psi_{1,i}(\bar{f}_2) b''_1) \bar{q}''_1.$$
From the definition of a superswitcher $b''_1$ we get that 
$q'_1  = \bar{q}'_1 \bar{f}_1 b'_1 \psi_{1,i}(\bar{f}_2),$
but this is impossible by definition of superwitcher $b'_1$. So we get that $\sigma = \bar{\sigma}$. 

Assume that $\sigma = \bar{\sigma} = +1$. 
We get 
$$q'_1 f_1 b'_1 \psi_{1,i}(f_2) b''_1 q''_1 = \bar{q}'_1 \bar{f}_1 b'_1 \psi_{1,i}(\bar{f}_2) b''_1 \bar{q}''_1,$$ so, since $b''_1$ is a superswitcher, 

$$q'_1 f_1 b'_1 \psi_{1,i}(f_2)  = \bar{q}'_1 \bar{f}_1 b'_1 \psi_{1,i}(\bar{f}_2),$$
and 
$$q''_1 = \bar{q}''_1.$$
Analogously,
$$q'_2 f_2 b'_2 \psi_{2,i}(f_1) b''_2 q''_2 = \bar{q}'_2 \bar{f}_2 b'_2 \psi_{2,i}(\bar{f}_1) b''_2 \bar{q}''_2,$$ so, since $b''_2$ is a superswitcher, 

$$q'_2 f_2 b'_2 \psi_{2,i}(f_1)  = \bar{q}'_2 \bar{f}_2 b'_2 \psi_{2,i}(\bar{f}_1),$$
and 
$$q''_2 = \bar{q}''_2.$$
Now using the superswitcher property for $b'_j$, we get 
$$\psi_{j,i}(f'_{2-j}) = \psi_{j,i}(\bar{f}'_{2-j}),$$
and
$$q'_j f'_j = \bar{q}'_j \bar{f}'_j.$$
Hence, $f'_j = f'_j$ and $f''_j = f''_j$ since functions $\psi_{j,i}$ are bijections, and $q'_j = \bar{q}'_j$, $q''_j = \bar{q}''_j$.

It is left to handle the case $\sigma = \bar{\sigma} = -1$, but this is done the similar way.
\end{proof}

We are ready to define function $p: \bigcup W_n \to \Gamma$.
For 
$$g = (q'_1 (f'_1 b'_1 \psi_{1,i}(f'_2, f''_2) b''_1 f''_1)^{\sigma} q''_1, \quad q'_2 (f'_2 b'_2 \psi_{2,i}(f'_1, f''_1) b''_2 f''_2)^{\sigma} q''_2),$$
where $q'_j \in A_{j,i}^{i+1}$, $f'_j, f''_j \in F_{j,i}$, $q'' \in A_{j,i}^i$,
we set $p(g) = (q'_1, q'_2)$. This is a well-defined on $\bigcup_n W_n$ function by the previous lemma and pairwise disjointness of sets $W_n$. 

In a similar fashion to the previous section, we can check the requirements of Lemma \ref{lem: main non-triviality}. So the Poison-Furstenberg boundary $\partial(\Gamma, \nu)$ is non-trivial and the action of $\Gamma$ on it is essentially free. Note, that we already proved that the boundaries $\partial(G_i, \pr_i(\nu))$, $i = 1,2$, are trivial.

\end{document}